\documentclass[12pt]{amsart}
\usepackage{amsmath,amsfonts,latexsym,graphicx,amssymb,url}
\usepackage{hyperref,pdfsync}
\usepackage{amsmath,txfonts,pifont,bbding,pxfonts,manfnt}
\usepackage[active]{srcltx}
\usepackage{wasysym,pstricks}

\newcommand{\sign}{\operatorname{sign}}
\newcommand{\Lip}{\operatorname{Lip}}

\newcommand{\R}{{\mathbb R}} 

\renewcommand{\S}{{\mathbb S}}

\newcommand{\N}{{\mathbb N}}
\newcommand{\Ha}{{\mathcal H}}
\renewcommand{\H}{{\mathbb H}}

\newtheorem{theorem}{Theorem}

\newtheorem{lemma}[theorem]{Lemma}
\newtheorem{proposition}[theorem]{Proposition}

\newtheorem{remark}[theorem]{Remark}

\begin{document}

\title[Lipschitz extensions of maps between Heisenberg groups --- \today]{Lipschitz extensions  of maps between  Heisenberg groups}
\author[Zolt\'an M.~Balogh, Urs Lang, Pierre Pansu --- \today]{Zolt\'an M.~Balogh, Urs Lang, Pierre Pansu}
\date{\today}

\thanks{Z.M.B.~is supported by MAnET Marie Curie Initial 
Training Network and the Swiss National Science Foundation, U.L.~is supported by the Swiss National Science Foundation, 
P.P.~is supported by MAnET Marie Curie Initial 
Training Network and Agence Nationale de la Recherche, ANR-10-BLAN 0116.}
\begin{abstract}
Let $\H^n$ be the Heisenberg group of topological dimension $2n+1$. We prove that if $n$ is odd, the pair of metric spaces $(\H^n, \H^n)$ does not have the Lipschitz extension property.
\end{abstract}
\address{Department of Mathematics, University of Bern, Sidlerstrasse 5, CH-3012  Bern,  Switzerland}
\email{zoltan.balogh@math.unibe.ch}
\address{Department of Mathematics, ETH Z\"urich, R\"amistrasse 101, CH-8092 Z\"rich, Switzerland}
\email{urs.lang@math.ethz.ch}
\address{Department of Mathematics, Universit\'e Paris-Sud, B\^atiment 405, 
F-91405, Orsay, France}
\email{pierre.pansu@u-psud.fr} 

\maketitle


\setcounter{equation}{0}

\section{Introduction} 

Since globally defined Lipschitz mappings have good Rademacher type differentiability properties, it is a matter of general interest to provide extension theorems for partially defined Lipschitz mappings. In this context, the following concept is useful: A pair of metric spaces $(X,Y)$ is said to have the {\it Lipschitz extension property} if there exists a constant $C\geq 1$ such that any Lipschitz mapping $f \colon A \to Y$, $A\subset X$, has a Lipschitz extension $F \colon X \to Y$ with the property that $\Lip(F) \leq C\cdot \Lip(f)$. Here, we used the notation $\Lip(F) := \inf \{ L\geq 0: d_Y(F(x),F(y)) \leq L d_X(x,y) \ \text{for all} \ x,y\in X \}$ and similarly for $\Lip(f)$.
 
There is a long history of results providing examples of metric pairs $(X,Y)$ with or without the Lipschitz extension property. The first such result is due to McShane~\cite{McShane} who proved that for an arbitrary metric space $X$, the pair $(X,\R)$ has the Lipschitz extension property, even with  constant $C=1$. This result was followed by the theorem of Kirszbraun~\cite{Kirszbraun} who proved the Lipschitz extension property, also with constant $C=1$, for pairs of Euclidean spaces $(\R^m, \R^n)$. Kirszbraun's theorem was generalized by Valentine~\cite{Valentine} to pairs of general Hilbert spaces $(H_1, H_2)$. 
 
The Lipschitz extension problem becomes more challenging if the target space $Y$ has no linear structure. An generalization of the Kirszbraun--Valentine theorem to metric spaces with a curvature bound in the sense of Alexandrov was established by Lang and Schroeder~\cite{Lang-Schroeder}, and the case of an arbitrary metric space $X$ and a Hadamard manifold $Y$ was discussed in~\cite{Lang-Pavlovic-Schroeder}. In the paper of Lang and Schlichenmaier~\cite{Lang-Schlichenmaier}, a general sufficient condition for the Lipschitz extension property of a metric pair $(X,Y)$ is given.

The Lipschitz extension property for a pair of general Carnot groups $(G_1, G_2)$ with their canonical sub-Riemannian Carnot-Carath\'eodory metrics is studied for particular pairs of Carnot groups in~\cite{Balogh-Faessler}, \cite{Rigot-Wenger}, \cite{Wenger-Young}. The situation is fully understood for the pair $(\R^n, \H^m)$, where $G_1= \R^n$ is the Euclidean space of dimension $n$ and $\H^m$ is the Heisenberg group of topological dimension $2m+1$. In this case, it is known that the pair $(\R^n, \H^m)$ has the Lipschitz extension property if and only if $n\leq m$. The "only if" part of this equivalence  is due to Balogh and F\"assler~\cite{Balogh-Faessler}, whereas the "if" part is due to Wenger and Young~\cite{Wenger-Young}. 

\medskip 
In this note, we shall study the Lipschitz extension property for pairs of Heisenberg groups $(\H^n, \H^m)$, a question that was posed in~\cite{Balogh-Faessler}. Let us observe first that since $\H^n$ contains isometric copies of $\R^n$, it follows from~\cite{Balogh-Faessler} that if $(\H^n, \H^m)$ has the Lipschitz extension property, then we must necessarily have that $n\leq m$. 

The reason of this restriction on dimensions is as follows: It is a well-known fact of sub-Riemannian geometry that smooth Lipschitz mappings between sub-Riemannian spaces must satisfy the {\it contact condition}. This means that the horizontal distribution defining the sub-Riemannian metric must be preserved by the tangent map. This property gives an obstruction for the existence of smooth Lipschitz extensions. This principle can be transferred even for possibly non-smooth Lipschitz extensions, allowing us to prove the main result of this note:

\begin{theorem} \label{main}
If $n \geq 2k+1$, then the pair $(\H^n,\H^{2k+1})$ does not have the Lipschitz extension property. 
\end{theorem}

Since $\H^{2k+1}$ is isometrically embedded in $\H^n$ for $2k+1 \leq n$ it is enough to prove the statement for $n=2k+1$. The idea of the proof is to exploit the consequences of the aforementioned contact condition. Indeed, this condition implies that a globally defined smooth contact map $F \colon \H^{2k+1} \to \H^{2k+1}$ is orientation preserving, which gives a major restriction on the nature of globally defined Lipschitz mappings. It allows us to define a map $f \colon H \to \H^{2k+1}$ from a subset $H \subset \H^{2k+1}$ into $\H^{2k+1}$ that does not have a globally defined Lipschitz extension. 
In order to prove the non-existence of such an extension we use results from mapping degree theory \cite{Lloyd}, \cite{Fonseca-Gangbo} and Pansu's result on almost everywhere differentiability \cite{Pansu}. 
This argument does not work for the case of the pairs $(\H^{2k}, \H^{2k})$ as smooth contact maps $F \colon \H^{2k} \to \H^{2k}$ are not necessarily orientation preserving. In this case the Lipschitz extension property is still open. 

\medskip
The paper is organized as follows: In Section 2 we recall the metric structure of the Heisenberg group and discuss regularity results of Lipschitz maps between Heisenberg groups to be used in the sequel. In Section 3 we prove a result on the computation of the mapping degree of maps between Heisenberg groups using the Pansu derivative. This result may be of independent interest. In Section 4 we give the proof of Theorem~\ref{main}. 

\medskip
{\bf Acknowledgements:} The authors thank Jeremy T.~Tyson, Stefan Wenger, Piotr Hajlasz and Enrico Le Donne for many stimulating conversations on the subject of the paper.

\section{The Heisenberg metric and regularity of Lipschitz maps between Heisenberg groups}

We recall background results on the geometry of the Heisenberg group and preliminary results on the differentiability properties of Lipschitz maps between Heisenberg groups. 

The model of the Heisenberg group $\H^n$ used in this paper is described as follows: $\H^n := \R^n\times \R^n \times \R$, endowed with the group operation
$$(x,y,t) \ast (x',y',t') = (x+x', y+y', t+t' +2(y\cdot x'-x\cdot y')),$$
for $(x,y,t), (x',y',t') \in \H^n$. Here, $x\cdot y$ stands for the usual scalar product of vectors $x,y\in \R^n$.

It is easy to check that the vector fields $X_i= \frac{\partial}{\partial x_i} + 2 y_i \frac{\partial}{\partial t} , Y_i= \frac{\partial}{\partial y_i} - 2 x_i \frac{\partial}{\partial t} $ are left invariant for $i = 1,\ldots,n$. For $p\in \H^n$ the space $H_p = span \{ X_i,Y_i : i = 1,\ldots,n \} $ is the horizontal plane at $p$. Varying $p\in \H^n$ gives rise to a vector bundle $\Ha$, 
the horizontal plane distribution. 

Let us choose the Riemannian metric on $\Ha$ which makes $\{ X_i,Y_i: i = 1,\ldots,n \} $ orthonormal at each point $p\in \H^n$ and denote the length of horizontal vectors $v\in H_p$ in this metric by $|v|$. 

A Lipschitz curve $\gamma \colon [0,1] \to \H^n$ is said to be horizontal if $\gamma'(t) \in H_{\gamma(t)}$ for a.e.~$t\in [0,1]$. For $p,q \in \H^n$ we denote by $d_{cc}(p,q)$ the sub-Riemannian Carnot-Carath\'eodory distance between $p$ and $q$ defined as
$$d_{cc}(p,q) := \inf \int_0^1|\gamma'(t)| dt,$$
where the $\inf$ is taken over all horizontal curves connecting $p$ and $q$. Such horizontal curves exist between any pair of points $p,q \in \H^n$. This follows from the fact that the horizontal distribution is bracket-generating. For an explicit computation of the geodesics and the exact formulae of the metric we refer to~\cite{Monti} (see also~\cite{CDPT} or~\cite{Ambrosio-Rigot}). 

It is easy to check that $d_{cc}$ is invariant under left translation. Moreover, it is compatible with the homogeneous dilations $\delta_s \colon \H^n \to \H^n$ defined by $\delta_s(x,y,t) = (sx,sy,s^2t)$, i.e.~$d_{cc}(\delta_s(p), \delta_s(q)) = s d_{cc}(p,q)$ for all $s > 0$ and $p,q \in \H^n$. 

In this paper the metric on $\H^n$ is understood to be the above defined Carnot-Carath\'eodory metric $d_{cc}$. However the main result of the paper remains valid for any metric $d$ on $\H^n$ that is bi-Lipschitz equivalent to $d_{cc}$. 


\medskip 
It is easy to see that a smooth Lipschitz map $f \colon \H^n \to \H^{n}$ has the property that its tangent map preserves the horizontal bundle $\Ha$. Since the horizontal bundle is the kernel of the one-form
$$\theta = dt-2 \sum_{i=1}^{2k+1}(y_{i}dx_{i}-x_{i}dy_{i}),$$
we obtain that there exists a real valued smooth function $\lambda$ such that $f^{\ast}\theta = \lambda \cdot \theta$.

If $f$ is merely Lipschitz with no additional smoothness assumptions, then $\lambda$ is only an $L^{\infty}$ function and the equality $f^{\ast}\theta = \lambda \cdot \theta$ holds almost everywhere. This follows from the following differentiability result: 

A Lipschitz map $f \colon \H^{n}\to \H^{n}$ is called $P$-differentiable at a point $p\in \H^{n}$ if there exists a group homomorphism $P(f)(p) \colon \H^{n} \to \H^{n}$ such that 
\begin{equation} \label{p-diff}
\lim_{s \to 0} \delta_{1/s}\left( (f(p))^{-1}\ast f(p\ast \delta_{s}q) \right) = P(f)(p)(q)
\end{equation}
for all $q\in \H^{n}$. Here $\ast$ stands for the group operation, $\delta_{s}$ is the homogeneous group dilation, and $p^{-1}$ is the inverse to $p$ in $\H^{n}$. Pansu proved in \cite{Pansu} that any Lipschitz map is $P$-differentiable almost everywhere, moreover in a point of $P$-differentiability the limit \eqref{p-diff} is uniform in $q$ if $q$ is taking its values in a compact set. 

Let us define the singular set of a Lipschitz map $f \colon \H^{n} \to \H^{n}$ as 
$$ S = S_{1}\cup S_{2},$$
$$ S_{1} := \{ p \in \H^{n}: \text{$f$ is not $P$-differentiable at $p$} \},$$
$$ S_{2} := \{ p \in \H^{n}: \text{$f$ is $P$-differentiable at $x$ but $P(f)(p)$ is not injective} \},$$
and the regular set as 
$$R = \H^{n}  \setminus S.$$ 

For $d>0$ we denote by $\Ha^{d}$ the $d$-dimensional Hausdorff measure with respect to the Heisenberg metric. By Pansu's theorem we have $\Ha^{2n+2}(S_{1}) = 0$, which implies by the Lipschitz continuity of $f$ that also $\Ha^{2n+2}(f(S_{1})) = 0$. Furthermore, we observe that if $p\in S_{2}$, the metric Jacobian $J_{f}(p)$ of $f$ at $p$ vanishes, where 
$$J_{f}(p) := \liminf_{r\to 0+}\frac{\Ha^{2n+2}(f(B(p,r)))}{\Ha^{2n+2}(B(p,r))} = 0$$
and $B(p,r)$ is the (open) ball with center $p$ and radius $r$ in the Heisenberg metric. For any measurable set $A\subset \H^{n}$ we have 
$$\Ha^{2n+2}(f(A)) \leq \int_{A}J_{f}(p) \,d\Ha^{2n+2}(p)$$
(see e.g.~Chapter~6 of~\cite{CDPT}), so $\Ha^{2n+2}(f(S_{2})) = 0$ and
$$\Ha^{2n+2}(f(S)) = 0.$$

\section{The mapping degree of Lipschitz maps in Heisenberg groups}

In this section we will use the $P$-derivative in order to estimate the mapping degree of Lipschitz maps between Heisenberg groups. For background results on degree theory we refer to \cite{Lloyd}, \cite{Fonseca-Gangbo}.

\medskip 
The considerations of this section are based on the fact that smooth contact maps  $f \colon \H^{2k+1}\to \H^{2k+1}$ have a positive Jacobi determinant, and so they preserve orientation. To see this, let 
$$\theta = dt-2 \sum_{i=1}^{2k+1}(y_{i}dx_{i}-x_{i}dy_{i})$$
be the contact form in $\H^{2k+1}$. Observe that if we denote by $v$ the volume form in $\H^{2k+1}$, then there exists a dimensional constant $c(k)>0$ such that $v = c(k)\cdot \theta \wedge (d\theta)^{2k+1}$.
If $f$ is a smooth contact mapping then $f^{\ast}\theta = \lambda \theta$ for some real valued smooth function $\lambda$. The relation between the Jacobi determinant of $f$ and  $\lambda$  can be computed as follows:
\begin{align*}
f^{\ast}v &= c(k)\cdot f^{\ast}\theta  \wedge (f^{\ast}d\theta)^{2k+1}= c(k) \cdot \lambda \theta \wedge (d (\lambda \theta ))^{2k+1} \\
&= c(k) \cdot \lambda \theta \wedge (d\lambda \wedge \theta + \lambda \,d\theta)^{2k+1}= 
c(k) \cdot \lambda^{2k+2} \theta \wedge (d \theta)^{2k+1} = \lambda^{2k+2} v,
\end{align*}
which implies that $\det Df = \lambda^{2k+2} \geq 0$. In particular, it follows that group homomorphisms $f \colon \H^{2k+1} \to \H^{2k+1}$ are orientation preserving. 

\medskip
We want to extend this observation from the class of smooth Lipschitz maps to possibly non-smooth Lipschitz mappings. This is possible by a version of the Rademacher differentiability theorem due to Pansu \cite{Pansu} according to which Lipschitz maps between Carnot groups are differentiable almost everywhere; and their tangent map satisfies the contact condition in the points of differentiability. Note that differentiability is understood here in terms of the underlying group operation and compatible homogeneous dilations as indicated in the previous section. The classical Rademacher differentiability theorem that is formulated in terms of vector addition and scalar multiplication as in the Euclidean space  fails for Lipschitz maps between Carnot groups. An example showing this fact is due to Magnani \cite{Magnani}.  Nevertheless, Pansu's differentiability theorem can be used to estimate the mapping degree for maps between Carnot groups.
The main result of this section is the following statement:

\begin{proposition} \label{degree} 
Let $k\geq 0$ be an integer, $f \colon \H^{2k+1}\to \H^{2k+1}$ a Lipschitz map, $U\subset \H^{2k+1}$ an open set, and $p\in \H^{2k+1}\setminus f(\partial U)$. If $p\notin f(S)$, then 
$$\deg (f,U,p) \geq 0.$$
\end{proposition}

\begin{proof} 
Let us consider a (possibly) non-smooth Lipschitz map $f \colon \H^{2k+1}\to \H^{2k+1}$ as in the statement. There are two cases to consider: If $p\notin f(U)$, then $\deg (f, U, p)= 0 $ by the definition of the degree (\cite{Lloyd}, \cite{Fonseca-Gangbo}) and we are done. Let us therefore assume that $p\in f(U) \setminus [f(\partial U) \cup f(S)]$. In this case we shall prove that 
$$ \deg (f, U, p) > 0.$$
The proof is based on the fact that for group homomorphisms this property holds true as indicated above. This fact will be combined with a localization argument using linearization by the $P$-differential of $f$ at regular points. 

The first step is to show that if $p\in f(U) \setminus [f(\partial U) \cup f(S)]$, the pre-image 
$$f^{-1}(p)\subset R \cap U$$
is a finite set. 

To prove this, assume by contradiction that there exists an infinite sequence $(p_{n})_{n}$ of distinct points in $R\cap U$ such that $f(p_{n}) = p$ for each $n\in \N$. By passing to a subsequence we can assume that $p_{n}\to p_{0}$ for some $p_{0} \in \overline{U}$. By continuity of $f$ we obtain that $f(p_{0})= p$. Since $p\notin f(\partial U) \cup f(S)$, this implies that $p_{0}\in R\cap U$. We shall obtain the desired contradiction by showing that $P(f)(p_{0})$ is non-injective. 

To see this, let us write $p_{n}= p_{0}\ast\delta_{t_{n}}q_{n}$ for some $q_{n} \in \partial B(0,1)$ and $t_{n} > 0$, $t_{n}\to 0$. By passing to a further subsequence if necessary we can assume that $q_{n}\to q$ for some $q\in \partial B(0,1)$. Since $f(p_{0}) = p = f(p_{n}) = f(p_{0}\ast\delta_{t_{n}}q_{n})$, we have
\begin{align*}
P(f)(p_{0})(q) 
&= \lim_{t_{n} \to 0} \delta_{1/t_{n}} 
\left( (f(p_{0}))^{-1} \ast f(p_{0} \ast \delta_{t_{n}}q) \right) \\
&= \lim_{t_{n} \to 0} \delta_{1/t_{n}} 
\left( (f(p_{0}\ast\delta_{t_{n}}q_{n}))^{-1} \ast f(p_{0} \ast \delta_{t_{n}}q) \right),
\end{align*}
but this last limit equals $0$ by the Lipschitz continuity of $f$ and the fact that $q_{n}\to q$. So $P(f)(p_{0})(q) = 0$, in contradiction to $p_0 \in R$.   

By the first step we can assume that 
$$ f^{-1} (p) = \{ p_{1}, \ldots , p_{n} \} \subset R \cap U .$$
By general properties of the mapping degree (\cite{Lloyd}, \cite{Fonseca-Gangbo}) we have that 
$$ \deg (f, U, p) = \sum_{i=1}^{n} \deg (f, U_{i}, p) ,$$
where $U_{i}$ are arbitrary small, non-intersecting open sets containing $p_{i}$ for $i= 1,\ldots, n$.

In the following we shall prove that $\deg(f,U_{i},p)=1$ for every $i$. The idea is to construct a homotopy from $f$ to its linearization $P(f)(p_{i})$ in $\overline{U_{i}}$ with the property that the image of $\partial U_{i}$ under the homotopy does not meet the point $p$.

Let us fix an index $i\in \{1, \ldots , n\}$. To simplify the notation, we can assume without loss of generality that $p=p_{i}=0$, $U_{i}= B(0,r)$ for some small $r>0$ (to be chosen later), and $P(f)(p_{i}) = P(f)(0)= A$, where $A \colon \H^{2k+1} \to \H^{2k+1}$ is an injective homogeneous group homomorphism (i.e.~commuting with dilations). 

The $P$-differentiability of $f$ at $0$ now reads as
$$ \lim_{s\to 0}\delta_{1/s}f(\delta_{s}q)= A(q),$$
and the convergence is uniform in $q\in \partial B(0,1)$. Equivalently, for every $\epsilon > 0$ there exists an $\eta > 0$ such that 
\begin{equation} \label{eq:eps-s}
d_{H}(f(\delta_sq),A(\delta_sq)) < \epsilon \cdot s
\end{equation}
for all $q \in \partial B(0,1)$ and $s < \eta$. Furthermore, by the injectivity and homogeneity  of $A$ there exists an $a > 0$ such that
\begin{equation} \label{eq:a-s}
d_{H}(A(\delta_sq),0) \geq a \cdot s
\end{equation}
for all $q\in \partial B(0,1)$ and $s > 0$. Let us choose $\epsilon:= a$ and $U_i = B(0,r)$ for some $r < \eta$.

Consider the homotopy $H \colon [0,1] \times \overline{B(0,r)} \to \H^{2k+1}$ given by 
$$H(s,q)= A(q)\ast \delta_{s}[(A(q))^{-1}\ast f(q)].$$
Clearly $H(0,\cdot)=A$ and $H(1,\cdot)=f$. For all $s \in [0,1]$ and $q\in \partial B(0,r)$ we have 
$$d_H(H(s,q),0) \ge d_H(A(q),0) - d_H(H(s,q),A(q)),$$
where $d_H(A(q),0) \ge a \cdot r$ by~\eqref{eq:a-s} and $d_H(H(s,q),A(q)) = s\,d_H(f(q),A(q)) < \epsilon \cdot r$ by~\eqref{eq:eps-s}, so that 
$$d_H(H(s,q),0) > 0.$$
By the homotopy invariance property of the mapping degree (\cite{Lloyd}, \cite{Fonseca-Gangbo}) we obtain that 
$$\deg(f,U_i,p) = \deg(f, B(0,r), 0) = \deg (A, B(0,r), 0) = 1,$$
finishing the proof. 
\end{proof} 

\begin{remark}
Observe that the statement of Proposition \ref{degree} no longer holds for maps $f \colon \H^{2k}\to \H^{2k}$. Indeed, the mapping $f \colon \H^{2k}\to \H^{2k}$ given by $f(x,y,t) = (-x,y,-t) $ is a contact map since $f^{\ast}\theta = - \theta$, and it is orientation reversing: $det D f \equiv -1$. More generally, it is easy to see that every integer is the degree of some contact map between domains in $\H^{2k}$.
\end{remark}

\section{A partially defined Lipschitz map in the Heisenberg group}

In this section we define a Lipschitz mapping on a subset $H$ of $\H^{n}$ to itself which does not have a Lipschitz extension to $\H^{n}$ in the case when $n=2k+1$. 

The set $H$ will be a topological $2n$-sphere foliated by geodesics connecting the origin $p_{0}$ to a point $p_{1}$ on the vertical axis. For simplicity of notation we choose $p_{1}=(0,0,4\pi) \in \R^{n}\times\R^{n}\times \R$. To define $H$ we use a parametrization $\Phi \colon \S^{2n-1}\times [0,2\pi] \to \H^{n}$, where
$$\S^{2n-1}:= \{ (a,b)\in \R^{n}\times \R^{n} : |a|^{2}+|b|^{2}=1 \}$$
is the unit sphere in $\R^{2n}$. The map $\Phi \colon \S^{2n-1}\times [0,2\pi] \to \H^{n}$ will be given by 
$$ \Phi((a,b),s) := \bigl( \sin(s)\cdot a + (1-\cos(s))\cdot b,\, 
-(1-\cos(s))\cdot a + \sin(s)\cdot b,\, 
2(s-\sin(s)) \bigr)$$ 
for $((a,b),s) \in \S^{2n-1}\times [0,2\pi]$. The hypersurface $H := \Phi(\S^{2n-1}\times [0,2\pi])$ is the boundary of a bounded domain $\Omega \subset \H^{n}$ and is foliated by the unit speed geodesics $s\to \Phi((a,b),s)$ connecting $p_{0}$ and $p_{1}$. For details we refer to the papers of Ambrosio-Rigot~\cite{Ambrosio-Rigot} and Monti~\cite{Monti}.
 
We claim that to each point $p\in H\setminus \{p_{0}, p_{1} \}$ there is a unique value $((a,b),s) \in \S^{2n-1}\times(0,2\pi)$ such that $\Phi((a,b),s)=p$, whereas $\Phi^{-1}(p_{0})= \S^{2n-1}\times\{0 \}$ and $\Phi^{-1}(p_{1})= \S^{2n-1}\times\{2\pi\}$. To verify this, let $(x,y,z) \in H\setminus \{ p_{0},p_{1}\}$. Then $z\in (0,4\pi)$, and since the function $\phi \colon (0,2\pi) \to (0,4\pi)$, $\phi(s)= 2(s-\sin(s))$ is strictly increasing and onto, we get a unique value $s\in (0,2\pi)$ such that $\phi(s) = z$. Furthermore, the system 
\begin{equation} \label{ab}
 \left \{
\begin{array}{rrcl}
\sin(s) \cdot a & + \,(1-\cos(s)) \cdot b & = & x \\
-(1-\cos(s)) \cdot a & + \,\sin(s) \cdot b & = & y   
\end{array}
\right.
\end{equation}
has a unique solution $(a,b)$, proving the claim.

Consider now the {\it symplectic reflection} $R \colon \R^{2n} \to \R^{2n}$ given by $R(a,b)=(-a,b)$. Observe that $R^{\ast}\omega = -\omega$ for the standard symplectic form $\omega = \sum_{i=1}^{n} dx_{i} \wedge dy_{i}$. We define the mapping $f \colon H \to H$ using $R$ as follows: If $p\in \{p_{0}, p_{1}\}$, set $f(p) = p$. If $p\in H\setminus \{ p_{0}, p_{1}\}$, let $((a,b),s) = \Phi^{-1}(p)$ and put
$$f(p) = \Phi(R(a,b),s).$$
Observe that the mapping interchanges the geodesics according to the reflection $(a,b)\mapsto R(a,b)$ while it preserves the {\it height} $s$ of points on $H$. 
 
\begin{lemma} \label{example} 
The mapping $f \colon H \to H$ defined above is a Lipschitz map with respect to the restriction of the sub-Riemannian Carnot-Carath\'eodory metric $d_{cc}$ on $H$ in both the source and target space. 
\end{lemma}
 
\begin{proof}
We have to find a constant $C>0$ such that 
$$ d_{cc}(f(p),f(p')) \leq C\, d_{cc}(p,p') \quad \text{for all $p,p' \in H$.}$$
Let $p,p' \in H$. We can assume that $p,p' \in H \setminus \{p_0,p_1\}$, so that there are unique preimages $((a,b),s) = \Phi^{-1}(p)$ and $((a',b'),s') = \Phi^{-1}(p')$. Consider a third point $q\in H$ lying on the same geodesic as $p'$ but having the height of $p$, that is, $q = \Phi((a',b'),s)$. By the triangle inequality we have 
$$d_{cc}(f(p), f(p')) \leq d_{cc}(f(p), f(q)) + d_{cc}(f(q),f(p')).$$
Since $q$ and $p'$ are on the same geodesic, the same is true for 
$f(q)$ and $f(p')$. It follows that 
$$d_{cc}(f(q),f(p')) = |s - s'| = |d_{cc}(p_0,p) - d_{cc}(p_0,p')| \le d_{cc}(p,p').$$
Hence we obtain
\begin{equation} \label{first-estimate}
d_{cc}(f(p),f(p')) \leq d_{cc}(f(p),f(q)) + d_{cc}(p,p').
\end{equation}  

We claim that there is a constant $C'>0$ such that 
\begin{equation} \label{same-heights}
d_{cc}(f(p),f(q)) \leq C' d_{cc}(p,q).
\end{equation} 
Here we will use in an essential way that $p$ and $q$ have the same heights. In order to verify the above estimate we shall consider another metric $d_{K}$ that is bi-Lipschitz equivalent to $d_{cc}$ and prove that 
\begin{equation} \label{koranyi}
d_{K}(f(p),f(q)) = d_{K}(p,q).
\end{equation} 
The  estimate~\eqref{same-heights} follows from~\eqref{koranyi} and the bi-Lipschitz equivalence of $d_{cc}$ and $d_{K}$. 

For two points $(x,y,t), (x',y',t') \in \H^{n}$ the Kor\'anyi gauge metric $d_{K}$ is given by the expression
$$d_{K}((x,y,t),(x',y',t')) = \bigl( ( |x-x'|^{2} + |y-y'|^{2})^{2}
+ |t-t' + 2(x\cdot y' - y \cdot x')|^{2} \bigr)^{1/4},$$
where $| \cdot |$ and $\cdot$ denote the Euclidean norm and scalar product in $\R^{n}$. It can be checked that this defines a metric indeed. This metric is no longer given as a sub-Riemannian metric but is bi-Lipschitz equivalent to $d_{cc}$. We refer to \cite{CDPT} for details. 

To verify equality~\eqref{koranyi} we shall use the explicit forms of $p$ and $q$:
\begin{align*}
p &= \bigl( \sin s \cdot a + (1-\cos s) \cdot b,\, 
-(1-\cos s)\cdot a + \sin s \cdot b,\, 2 (s-\sin s) \bigr), \\
q &= \bigl( \sin s \cdot a' + (1-\cos s) \cdot b',\, 
-(1-\cos s)\cdot a' + \sin s \cdot b',\, 2 (s-\sin s) \bigr),
\end{align*}
where $s\in (0,2\pi)$ and $(a,b),(a',b')\in \S^{2n-1}$. The respective expressions for $f(p)$ and $f(q)$ are obtained by replacing $a$ by $-a$ as well as $a'$ by $-a'$ and leaving $b$ and $b'$ unchanged. Now a direct computation gives
\begin{equation} \label{computation}
d_{K}(p,q) = (2 - 2 \cos s)^{1/2} \bigl( (|a-a'|^{2}+ |b-b'|^{2})^{2} 
+ 4|a\cdot b' - b\cdot a'|^{2} \bigr)^{1/4}.
\end{equation}
To calculate $d_{K}(f(p),f(q))$, all we need to do is to replace on the right side of~\eqref{computation} $a$ by $-a$ and $a'$ by $-a'$. After doing that, we see that its value will not change. This gives~\eqref{koranyi} and hence~\eqref{same-heights}. Notice that in this computation it was crucial that the two points $p$ and $q$ have the same heights. Otherwise, the second part of the right side of~\eqref{computation} will be different. 

Combining~\eqref{first-estimate} and~\eqref{same-heights} we obtain
$$d_{cc}(f(p), f(p')) \leq C'd_{cc}(p,q) + d_{cc}(p,p').$$
As $d_{cc}(p,q) \leq d_{cc}(p,p') + d_{cc}(p',q) \leq 2d_{cc}(p,p')$,
this gives
$$ d_{cc}(f(p),f(p')) \leq (2C'+1) \,d_{cc}(p,p'),$$
concluding the proof of the Lemma.
\end{proof}
    
\medskip 
We are now ready to prove our main result about the existence of partially defined Lipschitz maps on $\H^{2k+1}$ with no Lipschitz extension to the whole Heisenberg group $\H^{2k+1}$. This follows from the following:
  
\begin{theorem} \label{non-extension} If $n=2k+1$, then the mapping $f \colon H \to H$ as in Lemma \ref{example} has no Lipschitz extension to the whole Heisenberg group, i.e.~there is no Lipschitz map $F \colon \H^{n}\to \H^{n}$ such that $F|_{H} = f$. 
\end{theorem}
 
\begin{proof} 
By contradiction assume that there is a Lipschitz extension $F \colon \H^{n} \to \H^{n}$ such that $F|_{H} = f$. Proposition~\ref{degree} implies that if $n=2k+1$, then 
$$ \deg (U, F, p) \geq 0 \ \text{for any} \ p \notin F(\partial U).$$
Let us choose $U = \Omega$ such that $\partial U = H$ and let $p$ be any point in the interior of $U$, for instance, take $p= (0,0,2\pi)$. Since $F$ is a reflection on the boundary we have 
$$ \deg (U, F, p) = -1,$$
which gives a contradiction, proving the statement. 

To check that $\deg(U, F, p) = -1$ we use the fact that $\deg(U, F, p) = \deg(U,\tilde{F},p)$ for any continuous map $\tilde{F} \colon \overline{U} \to \R^{2n+1}$ with the property that $\tilde{F}_{\partial U} = f$. 
 
We will use the relation~\eqref{ab} first to find an explicit formula for the mapping $f \colon H \to H$ in cartesian coordinates as
\begin{equation} \label{cartesian}
f(x,y,t) = 
\left[
\begin{array}{c}
  \cos \varphi(t) \cdot x + \sin \varphi(t) \cdot y   \\
  \sin \varphi(t) \cdot x - \cos \varphi(t) \cdot y   \\
  t 
\end{array}
\right],
\end{equation} 
where $\varphi \colon [0,4\pi] \to [0,2\pi]$ is the inverse function of $\phi \colon [0,2\pi] \to [0,4\pi]$, $\phi(s) = 2 (s-\sin s)$.

Observe now that $\varphi \colon (0,4\pi) \to (0,2\pi)$ is a smooth function and formula \eqref{cartesian} itself in fact gives a continuous extension $\tilde{F} \colon \overline{U} \to \overline{U}$ that is smooth in $U$. The above formula indicates that $\tilde{F}$ preserves horizontal planes. In horizontal planes the map looks like a reflection followed by a rotation where the rotation angle depends on the height of the horizontal plane. 

Since $ \deg (U, F, p)$ is independent on the choice of the extension we have $ \deg (U, F, p)= \deg (U,\tilde{F},p)$. Because the mapping $\tilde{F}$ is smooth in $U$ we can calculate the degree by an explicit computation as 
$$\deg (U, \tilde{F}, p) = \sign \det D \tilde{F} (\tilde{F}^{-1}(p)) =  \sign \det D \tilde{F} (0,0,2\pi)=(-1)^{n}.$$
In the case when $n=2k+1$ we obtain that $\deg (U, \tilde{F}, p)= -1$, finishing the proof. 
\end{proof} 
  
In conclusion, let us come back to the general problem of determining exactly the pairs of Heisenberg groups $(\H^n, \H^m)$ with the Lipschitz extension property. Consider the case $n=1$. In~\cite{Wenger-Young} it is shown that the pair $(\H^1, \H^3)$ does have the Lipschitz extension property. On the other hand, by Theorem~\ref{main} this is not case for $(\H^1, \H^1)$. Therefore it only remains to study the extension problem for the pair $(\H^1, \H^2)$. We believe that understanding the Lipschitz extension property in this particular situation will shed light also on the general case of pairs $(\H^n, \H^m)$.

\bibliographystyle{alpha}

\end{document}